\documentclass[a5paper, 10pt]{article}

\usepackage{tikz,relsize,amsmath,amssymb,amsfonts,amscd,cite,booktabs,relsize}
\usepackage{lastpage,graphicx,eucal,adjustbox}
\usepackage[hidelinks]{hyperref}
\usepackage{subfigure}
\usepackage{geometry}
\usepackage{lipsum}
\usepackage{setspace}

\usepackage{amsthm}

\theoremstyle{plain}
\newtheorem{lemma}{Lemma}
\newtheorem{theorem}{Theorem}

\newtheorem{observation}[theorem]{Observation}

\theoremstyle{remark}

\theoremstyle{definition}

\usetikzlibrary{decorations.pathreplacing}
\allowdisplaybreaks

\usepackage{graphicx}
\makeatletter
\renewcommand{\maketitle}{
	\begin{center}

		\baselineskip=0.30in
		{\Large\bfseries \@title} \par
		\vspace{5mm}
		\baselineskip=0.2in
		{\large\bfseries \@author}\par
		\vspace{1mm}
		{\it \@address} \par
		{\small\tt \@email} \par
		\vspace{3mm}
		{\small } \par
	\end{center}
	\vspace{3mm}
}
\newcommand{\address}[1]{\def\@address{#1}}
\newcommand{\email}[1]{\def\@email{#1}}

\address{}
\makeatother

\newcommand{\acknowledgment}[1]{\vspace{5mm}\singlespacing
	{\noindent\textbf{\textit{Acknowledgment\/}:} #1}
}

\usepackage{fancyhdr}

\usepackage[margin=1cm,%
			font=footnotesize,%
			format=hang,%
			labelsep=period,%
			labelfont=bf]{caption}

\title{On the Diminished Sombor Index of Fixed-Order Molecular Graphs With Cyclomatic Number at Least 3}
\author{Abdulaziz Mutlaq Alotaibi$^{1}$, Abdulaziz M. Alanazi$^{2}$, Taher S. Hassan$^{3,4}$, Akbar Ali$^{3,}$\footnote{Corresponding author.}}

\address{$^1$Department of Mathematics,\\ College of Science and Humanities in AlKharj,\\ Prince Sattam Bin Abdulaziz University, AlKharj 11942, Saudi Arabia\\
$^2$Department of Mathematics, Faculty of Sciences,\\ University of Tabuk, Tabuk, Saudi Arabia\\
$^3$Department of Mathematics,  College of Science,\\ University of Ha\!'il, Ha\!'il 2440, Saudi Arabia\\
$^{4}$Department of Mathematics, Faculty of Science\\ Mansoura University, Mansoura 35516, Egypt
}

\email{am.alotaibi@psau.edu.sa, am.alenezi@ut.edu.sa, tshassan@yahoo.com, akbarali.maths@gmail.com}

\begin{document}

\maketitle

\begin{abstract}
For a graph $G$ with edge set $E$, let $d(u)$ denote the degree of a vertex $u$ in $G$. The diminished Sombor (DSO) index of $G$ is defined as $DSO(G)=\sum_{uv\in E}\sqrt{(d(u))^2+(d(v))^2}(d(u)+d(v))^{-1}$. The cyclomatic number of a graph is the smallest number of edges whose removal makes the graph acyclic. A connected graph of maximum degree at most $4$ is known as a molecular graph.
The primary motivation of the present study comes from a conjecture, concerning the minimum DSO index of fixed-order connected graphs with cyclomatic number $3$, posed in the recent paper [F. Movahedi, I. Gutman, I. Red\v{z}epovi\'c, B. Furtula,  Diminished Sombor index, {\it MATCH Commun. Comput. Chem.\/} {\bf 95} (2026) 141--162]. The present paper gives all graphs minimizing the DSO index among all molecular graphs of order $n$ with cyclomatic number $\ell$, provided that $n\ge 2(\ell-1)\ge4$.
\end{abstract}

\onehalfspacing

\section{Introduction}

Chemical graph theory  \cite{Leite-24}, a branch of mathematical chemistry, employs graph-theoretical concepts and techniques to model and analyze molecular structures. (The graph-theoretical and chemical graph-theoretical terminology, used in this study, but not defined here, can be found in~\cite{Bondy-book,Chartrand-16} and~\cite{Wagner-18,Trina-book}, respectively.) In this framework, molecules are represented by graphs in which atoms correspond to vertices and chemical bonds to edges, providing a rigorous mathematical basis for investigating molecular properties. In graph-theoretical language, a connected graph of maximum degree at most $4$ is known as a molecular graph.
(Sometimes, molecular graphs are defined as the connected graphs of maximum degree at most 3; e.g., see \cite{Bonte-25}.)

Molecular descriptors are fundamental tools for virtual screening of molecular libraries and for predicting the physicochemical properties of molecules~\cite{Basak-book}. According to Todeschini and Consonni~\cite{Todeschini-20-book}, a molecular descriptor is ``the final result of a logical and mathematical procedure that transforms chemical information encoded in a symbolic representation of a molecule into a useful number or the result of some standardized experiment.'' When such descriptors are defined using molecular graphs, they are commonly referred to as topological indices in chemical graph theory.  For details on the chemical applications of topological indices, readers may consult the recent works~\cite{Desmecht-JCIM-24,Leite-24}.

Among these indices, degree-based topological indices~\cite{Ali-MMN-22,KD1,gut04,gutman13degree,Ali-CJC-16,borovicanin17zagreb,Ali-IJACM-17,Ali-IJQC-17,nithyaa24} play a particularly prominent role. Owing to their computational efficiency and predictive power, degree-based indices are fundamental tools in quantitative structure-property relationship (QSPR) studies, which facilitate the design and analysis of new chemical compounds.

Among the many degree-based topological indices, one of the widely studied is the Sombor index, introduced by Gutman \cite{gutman21geo}. For a graph $G$, it is defined as
\[
\mathcal{SO}(G) = \sum_{uv \in E(G)} \sqrt{d(u)^2 + d(v)^2},
\]
where $E(G)$ denotes the edge set of $G$, and $d(u)$ and $d(v)$ represent the degrees of the vertices $u$ and $v$, respectively. (When more than one graph is under discussion simultaneously, we write $d_G(u)$ to specify the degree of vertex $u$ in the graph $G$.)
The Sombor index has attracted considerable research attention; e.g., see the survey~\cite{liu22review} and some recent publications \cite{chen24,das24open,new2,new3,new11}.

In \cite{gutman21geo}, two variants of the Sombor index were considered, namely the reduced and average Sombor indices. Since then, numerous further modifications have been proposed, including the elliptic Sombor~\cite{gutman24elliptic,gutman24eu}, Euler--Sombor~\cite{albalahi25tricyclic,new7}, Zagreb--Sombor~\cite{Ali-AIMS-25,Ali-AIMS-2024}, diminished Sombor~\cite{Movahedi-26}, and augmented Sombor~\cite{ASO-index} indices. The present paper is concerned with the diminished Sombor (DSO) index. For a graph $G$, the DSO index is defined by
\[
\mathcal{DSO}(G)=\sum\limits_{uv\in E(G)}\frac{\sqrt{(d(u))^2+(d(v))^2}}{d(u)+d(v)}.
\]

The cyclomatic number of a graph is the smallest number of edges whose removal makes the graph acyclic.
The primary motivation of the present study comes from a conjecture, concerning the graphs minimizing the DSO index among all fixed-order connected graphs with cyclomatic number $3$, posed in the recent paper \cite{Movahedi-26}. According to this conjecture, the aforementioned extremal graphs minimizing the DSO index ``are those obtained by connecting two disjoint cycles by two edges, so that a quadrangle is formed.'' In this paper, it is shown that the graphs minimizing the DSO index among all molecular graphs of order $n$ with cyclomatic number $\ell$ are either $3$-regular graphs, or graphs with maximum degree $3$ and minimum degree $2$ in which the number of edges connecting the vertices of degrees $2$ and $3$ is $2$, provided that $n \geq 2(\ell-1) \geq 4$. Consequently, the structures of the extremal graphs conjectured in \cite{Movahedi-26} differ slightly from the actual ones.

We end this introductory section with the remark that the literature contains many general results about degree-based topological indices, in which the extremal graphs are similar to the ones obtained in the present paper; e.g., see the recent papers \cite{Ali-AIMS-M-25,Su-26} as well as Theorem 5.5, Theorems 5.9--5.17, and Theorem 5.23 in the recent survey paper \cite{Ali-BID-survey}. However, none of these general results is applicable to the DSO index.

\section{Results}

We start this section by recalling some definitions. Let $G$ be a graph with vertex set $V(G)$ such that $|V(G)|=n$. The graph $G$ is known as an $n$-order graph. For a vertex $u \in V(G)$, define  $N_G(u) = \{v \in V(G) : uv \in E(G)\}$. The elements of $N_G(u)$ are called neighbors of $u$ in $G$.  A vertex of degree one (at least three, respectively) is called a pendant vertex (branching vertex, respectively). An edge incident with a pendent vertex is called a pendent edge.
The minimum and maximum degrees of a graph $G$ are denoted by $\delta(G)$ and $\Delta(G)$, respectively.

Now, we recall the definition of circulant graphs (see, e.g.,
\cite{Cichacz-DM-16,Vetrik-CMB-16}).
Let $r$, $k$ and $a_1, a_2, \dots, a_k$ be positive integers such
that $$1 \le a_1< a_2< \dots< a_k \le \lfloor r/2\rfloor.$$
The circulant graph, denoted by $C(r;a_1, a_2, \dots, a_k)$, is the graph
with vertex set $\{v_0,v_1, v_2, \dots, v_{r-1}\}$ and edge set $$\{v_iv_{i+a_j}: 0\le i\le r-1 \text{ and } 1\le j\le k, \text{ where $i+a_j$ is taken modulo
$r$}\}.$$

\begin{observation}[e.g., see \cite{Boesch-JGT-84,Vetrik-CMB-16}]\label{obs-00}
If $a_k <   r/2$, then  $C(r;a_1, a_2, \dots, a_k)$ is $2k$-regular.
\end{observation}

\begin{lemma}{\rm\cite{Boesch-JGT-84}}\label{lem-00}
    The circulant graph $C(r;a_1, a_2, \dots, a_k)$ is connected if and only if $\gcd(a_1, a_2, \dots, a_k,r) = 1$.
\end{lemma}

\begin{lemma}\label{lem-01}
Let $G$ be a graph minimizing the DSO index among all molecular $n$-order graphs with cyclomatic number $\ell$ such that $n \ge 2(\ell-1)\ge 4$. Then, $G$ contains no pendent edge incident with a vertex of degree $2$.

\end{lemma}

\begin{proof}
We suppose to the contrary that $G$ contains at least one pendent edge incident with a vertex of degree $2$, say $uv\in E(G)$, where $d_G(u)=1$ and $d_G(v)=2$. We consider a branching vertex $w\in V(G)$ in such a way that the distance between $w$ and $u$ is the least. Then, every vertex different from $u$ and $w$ on the unique $u-w$ path, say $P$, has degree $2$ in $G$. Since $\ell\ge3$, $w$ has at least two non-pendent neighbors. Let $N_G(w)=\{w_1,w_2,\dots,w_s\}$, where $w_s$ lies on $P$ and $s\in \{3,4\}$. We may assume that $d_G(w_1)=\max\{d_G(w_i):1\le i\le s-1\}$. Here, we define  $f(i,j):=\sqrt{i^2+j^2}/(i+j)$.\\[2mm]
{\bf Case 1.} Either $d_G(w)=4$, or $d_G(w)=3$ and $(d_G(w_1),d_G(w_2))\ne(4,4)$.\\
We form a new graph $G_1$ using $G$ by removing the edge $w_1w$ and inserting the edge $w_1u$. Then
\begin{align}\label{lem-01-eq-01}
 \mathcal{DSO}(G)-\mathcal{DSO}(G_1) &=\sum_{i=2}^{s-1} \big[f(d_G(w),d_G(w_i))-f(d_G(w)-1,d_G(w_i)) \big]\nonumber\\[2mm]
 &\quad + f(d_G(w),2)-f(d_G(w)-1,2)+f(1,2)-f(2,2)\nonumber\\[2mm]
 &\quad +f(d_G(w),d_G(w_1))-f(2,d_G(w_1)).
  \end{align}
If $d_G(w)=4$ (that is, $s=4$), then using the facts that  $d_G(w_1)\ge 2$ and $1\le d_G(w_i)\le d_G(w_1)\le 4$ for each $i\in\{2,3\}$, we obtain from \eqref{lem-01-eq-01},  $\mathcal{DSO}(G)-\mathcal{DSO}(G_1)>0$, a contradiction to the definition of $G$.

If $d_G(w)=3$ (that is, $s=3$) and $(d_G(w_1),d_G(w_2))\ne(4,4)$, then using the facts that  $d_G(w_1)\ge 2$ and $1\le d_G(w_2)\le d_G(w_1)\le 4$, we again obtain from \eqref{lem-01-eq-01},  $\mathcal{DSO}(G)-\mathcal{DSO}(G_1)>0$, a contradiction.\\[2mm]
{\bf Case 2.}  $d_G(w)=3$ and $(d_G(w_1),d_G(w_2))=(4,4)$.\\
{\bf Case 2.1.} Either $w_1$ or $w_2$ does not have three neighbors of degree $4$.\\
Without loss of generality, we assume that $w_1$ does not have three neighbors of degree $4$. We choose a vertex $w_1'\in N_G(w_1)\setminus\{w\}$ in such a way that it satisfies the equation $d_G(w_1')=\min\{d_G(w'):w'\in N_G(w_1)\setminus\{w\}\}$. Then, $d_G(w_1')\le3$. Now, we form a new graph $G_2$ using $G$ by deleting the edge $w_1'w_1$ and inserting the edge $w_1'u$. Then, we have
\begin{align*}
 \mathcal{DSO}(G)-\mathcal{DSO}(G_2) &=\sum_{w'\in N_G(w_1)\setminus\{w,w_1'\}} \big[f(d_G(w'),4)-f(d_G(w'),3) \big]\nonumber\\[2mm]
 &\quad + f(d_G(w_1'),4)-f(d_G(w_1'),2)+f(1,2)-f(2,2)\nonumber\\[2mm]
 &\quad +f(3,4)-f(3,3)>0,
  \end{align*}
a contradiction.\\[2mm]
{\bf Case 2.2.} Both $w_1$ and $w_2$ have three neighbors of degree $4$.\\
We consider the graph $G^*$ obtained from $G$ by removing all the vertices of $P$ except $w$ and the edges incident with them. Then, $d_{G^*}(w)=2$ and $d_{G^*}(x)=d_{G}(x)$ for every $x\in V(G)\setminus V(P)$, where $V(P)$ is the set of all vertices belonging to the path $P$. In order to complete the discussion of Case 2.2 (and the proof of the lemma), we first prove a claim. \\[2mm]
{\bf Claim 1.} There exists at least one vertex $y\in V(G^*)\setminus\{w\}$ such that $1\le d_G(y)\le 3$.\\[2mm]
{\bf Proof of Claim 1.} Contrarily, assume that every $y\in V(G^*)\setminus\{w\}$ satisfies $d_G(y)=4$, that is, $d_{G^*}(y)=4$. We note that both the graphs $G$ and $G^*$ have the same cyclomatic number, namely, $\ell$. Hence, by the Handshaking Lemma, we have
\begin{equation}\label{eq-ell}
4(|V(G^*)|-1)+2=2|E(G^*)|=2(|V(G^*)|+\ell-1).
\end{equation}
If $\ell\le 5$, then by Equation \eqref{eq-ell}, we have
$|V(G^*)|\le 5,$
which contradicts the fact that $d_{G^*}(w)=2$ because  every $y\in V(G^*)\setminus\{w\}$ satisfies $d_{G^*}(y)=4$, according to our assumption. Next, we consider the case where $\ell\ge 6$. Then, by Equation \eqref{eq-ell}, we have
$|V(G^*)|\ge 6$. Now, we show that for every $\ell\,(\ge 6)$, there exists at least one connected graph with order $\ell$ and degree sequence $(\,\underbrace{4,4,\dots,4}_{\ell-1},2)$. By Observation \ref{obs-00}, the circulant graph $C(\ell-1;1, 2)$, with $\ell\ge6$, is $4$-regular. Also, by Lemma \ref{lem-00}, $C(\ell-1;1, 2)$ is connected. Let $C^\star(\ell-1;1, 2)$ denote the graph formed from $C(\ell-1;1, 2)$ by removing an edge $v_0v_1$ and adding a new vertex $v'$ together with the edges $v'v_0$
and $v'v_1$. Certainly, the order and the degree sequence of the graph $C^\star(\ell-1;1, 2)$ are $\ell$ and $(\,\underbrace{4,4,\dots,4}_{\ell-1},2)$, respectively. Thus, by keeping in mind the above discussion, we have $$m_{1,2}(G)=1=m_{2,3}(G),\ m_{3,4}(G)=2,\ m_{4,4}(G)=2\ell-3$$
and $m_{2,2}(G)=n-\ell-2$, where $m_{i,j}(G)$ represents the number of edges of $G$ with end vertices of degrees $i$ and $j$. Consequently, we have
\[
\mathcal{DSO}(G)=(n-\ell-2)f(2,2)+(2\ell-3)f(4,4)+f(1,2)+f(2,3)+2f(3,4).
\]
Let $G^\ddag$ be the $n$-order molecular graph with cyclomatic number $\ell$ such that  $n\ge 2(\ell-1)\ge4$
and $$m_{2,3}(G^\ddag)=2,\ m_{3,3}(G^\ddag)=3\ell - 4, \ m_{2,2}(G^\ddag)= n - 2\ell + 1.
$$

\begin{figure}[!ht]
 \centering
  \includegraphics[width=0.65\textwidth]{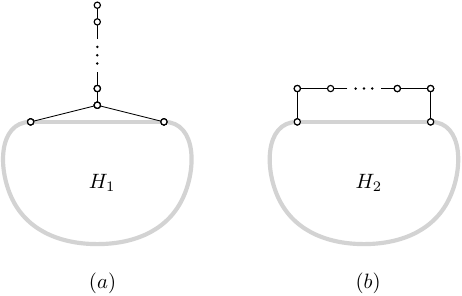}
   \caption{Two graphs used in the proof of Claim 1 of Lemma \ref{lem-01}.}
    \label{Fig-1}
     \end{figure}

\noindent
Figure \ref{Fig-1}(a) depicts $G$ (where every vertex in its subgraph $H_1$ has a degree $4$ in $G$), whereas Figure \ref{Fig-1}(b) shows the graph $G^\ddag$ (where every vertex in its subgraph $H_2$ has a degree $3$ in $G^\ddag$). Here, we have \[
\mathcal{DSO}(G^\ddag)=(n-2\ell+1)f(2,2)+(3\ell-4)f(3,3)+2f(2,3).
\]
Hence,
the difference $\mathcal{DSO}(G)-\mathcal{DSO}(G^\ddag)$ is equal to
\[
(\ell-3)f(2,2)+(2\ell-3)f(4,4)+f(1,2)-f(2,3)+2f(3,4)-(3\ell-4)f(3,3),
\]
which (is independent of $\ell$ because $f(2,2)=f(3,3)=f(4,4)$ and) is positive. This contradicts the definition of $G$. This completes the proof of Claim 1.\\[2mm]
\indent Next, using Claim 1, we discuss several possibilities. \\[2mm]
{\bf Case 2.2.1.} There exists a vertex $p\in V(G^*)\setminus\{w\}$ such that $d_G(p)=1$.\\
Let $p_1$ be the unique neighbor of $p$ in $G$.
We form a new graph $G_3$ using $G$ by deleting the edges $w_1w,w_2w,$ and inserting the edges $w_1w_2,wp$. Then, we have
\begin{align*}
 \mathcal{DSO}(G)-\mathcal{DSO}(G_3) &=f(d_G(p_1),1)-f(d_G(p_1),2) \nonumber\\[2mm]
 &\quad +f(2,3)+2f(3,4)- f(2,2)\\[2mm]
 &\quad-f(2,2)-f(4,4)>0,
  \end{align*}
a contradiction.\\[2mm]
{\bf Case 2.2.2.} There does not exists any vertex $p\in V(G^*)\setminus\{w\}$ satisfying $d_G(p)=1$, but there is a vertex $q\in V(G^*)\setminus\{w\}$ such that $d_G(q)=2$.\\
Since $G$ is connected, we choose $q$ in such a way that at least one of its neighbors has degree at least $3$ in $G$. Let $q_1$ and $q_2$ be the neighbors of $q$ such that $d_G(q_1)\ge3$ and $d_G(q_2)\ge2$.
We form a new graph $G_4$ using $G$ by deleting the edges $w_1w,w_2w,$ and inserting the edges $w_1w_2,wq$. Then, we have
\begin{align*}
 \mathcal{DSO}(G)-\mathcal{DSO}(G_4) &=\sum_{i=1}^2 \big[f(d_G(q_i),2)-f(d_G(q_i),3) \big]\nonumber\\[2mm]
 &\quad +2f(3,4)-f(2,2)-f(4,4)>0,
  \end{align*}
a contradiction.\\[2mm]
{\bf Case 2.2.3.} There does not exist any vertex $p\in V(G^*)\setminus\{w\}$ such that $d_G(p)=1$ or $d_G(p)=2$, but there is a vertex $z\in V(G^*)\setminus\{w\}$ such that $d_G(z)=3$.\\
Since $G$ is connected, we choose $z$ in such a way that at least one of its neighbors has degree $4$ in $G$. Let $z_1$, $z_2$ and $z_3$ be the neighbors of $z$ such that $d_G(z_1)=4$, $d_G(z_2)\ge3$ and $d_G(z_3)\ge3$.\\[2mm]
{\bf Case 2.2.3.1.} $d_G(z_2)=d_G(z_3)=4$.\\
We form a new graph $G_5$ using $G$ by deleting the edges $w_1w,w_2w,$ and inserting the edges $w_1w_2,wz$. Then, we have
\begin{align*}
 \mathcal{DSO}(G)-\mathcal{DSO}(G_5) &=5f(4,3)-4f(4,4)+f(2,3)\\[2mm]
 &\quad - f(2,4) -f(2,2)>0,
  \end{align*}
a contradiction.\\[2mm]
{\bf Case 2.2.3.2.} $\min\{d_G(z_2),d_G(z_3)\}= 3$.\\
Without loss of generality, we assume that $d_G(z_2)=3$. Then, either $d_G(z_3)=3$ or $d_G(z_3)=4$.
Now, we form a new graph $G_6$ using $G$ by deleting the edges $w_1w,w_2w,z_2z,$ and inserting the edges $w_1w_2,wz,z_2u$. Then, we have
\begin{align*}
 \mathcal{DSO}(G)-\mathcal{DSO}(G_6) &=f(1,2)+2f(3,4) +f(3,3)\\[2mm]
 &\quad-2f(2,2)- f(4,4)-f(2,3)>0,
  \end{align*}
a contradiction.
\end{proof}

\begin{theorem}\label{thm-1}
Let $G$ be a molecular $n$-order graph with cyclomatic number $\ell$ such that  $n \geq 2 (\ell-1)\ge 4$.
Then,
\begin{equation}\label{thm-1-Eq-stat}
\mathcal{DSO}(G) \ge \frac{n+\ell-3}{\sqrt{2}}+\frac{2\sqrt{13}}{5}.
\end{equation}
If $n=2(\ell-1)$, then the equality in \eqref{thm-1-Eq-stat} holds if and only if $G$ is $3$-regular. If $n>2(\ell-1)$, then the equality in \eqref{thm-1-Eq-stat} holds if and only if $\delta(G)=2$, $\Delta(G)=3$, $m_{3,3}(G)= 3\ell-4$, $m_{2,3}(G)=2$ and $m_{2,2}(G)=n-2\ell+1$. Examples of graphs satisfying the equality in \eqref{thm-1-Eq-stat} are given in Figure \ref{Fig-2}.
\end{theorem}

\begin{figure}[!ht]
 \centering
  \includegraphics[width=0.65\textwidth]{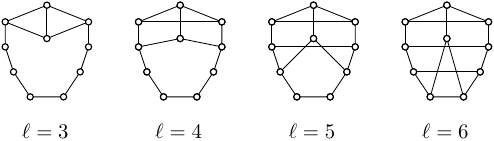}
   \caption{Examples of $10$-order graphs satisfying the equality in \eqref{thm-1-Eq-stat}.}
    \label{Fig-2}
     \end{figure}

\begin{proof}
Let $G^*$ be a graph minimizing the DSO index among all molecular $n$-order graphs with cyclomatic number $\ell$ such that $n \ge 2(\ell-1)\ge 4$. Then,
\begin{equation}\label{thm-1-eq-01}
\mathcal{DSO}(G)\ge \mathcal{DSO}(G^*).
\end{equation}
Also, by Lemma \ref{lem-01}, $G^*$ contains no pendent edge incident with a vertex of degree $2$; that is, $m_{1,2}(G^*)=0$.  In the rest of the proof, we use the notations $m_{i,j}:=m_{i,j}(G^*)$ and $$A:=\big\{(i,j): 1\le i\le j\le 4 \text{ and } i,j\in \mathbb{Z}_+\big\}\setminus\{(1,1),(1,2)\},$$
where $\mathbb{Z}_+$ represents the set of positive integers. Here, we have
\begin{equation*}
    \sum_{(i,j)\in A}\left(\frac{1}{i}+\frac{1}{j}\right)m_{i,j}=|V(G^*)|=n,
\end{equation*}
which yields
\begin{equation}\label{thm-eq-02}
    \sum_{(i,j)\in A^*}\left(\frac{1}{i}+\frac{1}{j}\right)m_{i,j}=n-m_{2,2}-\frac{5}{6}m_{2,3}-\frac{2}{3}m_{3,3},
\end{equation}
where $A^*:=A\setminus\{(2,2),(2,3),(3,3)\}$.
Also, the identity $$    \sum_{(i,j)\in A}m_{i,j}=n+\ell-1$$ gives
\begin{equation}\label{thm-eq-03}
    \sum_{(i,j)\in A^*}m_{i,j}=n+\ell-m_{2,2}-m_{2,3}-m_{3,3}-1.
\end{equation}
By solving \eqref{thm-eq-02} and \eqref{thm-eq-03} for $m_{2,3}$ and $m_{3,3}$, we obtain
\begin{equation}\label{thm-1-eq-m23}
m_{2,3}=2(n-2\ell- m_{2,2}+2)+\sum_{(i, j) \in A^*}\left(4-6 \left(\frac{1}{i}+\frac{1}{j}\right)\right) m_{i,j}
\end{equation}
and
\begin{equation}\label{thm-1-eq-m33}
m_{3,3}=5(\ell-1)+m_{2,2}-n+\sum_{(i,j) \in A^*}\left(6 \left(\frac{1}{i}+\frac{1}{j}\right)-5\right) m_{i,j}
\end{equation}
Now, using \eqref{thm-1-eq-m23} and \eqref{thm-1-eq-m33} in the formula of $\mathcal{DSO}(G^*)$, we obtain
\begin{align}
&\mathcal{DSO}(G^*)=\sum_{(i,j)\in A} \frac{\sqrt{i^2+j^2}}{i+j}m_{i,j}\nonumber\\[2mm]
&=\frac{1}{\sqrt{2}}(m_{2,2}+m_{3,3})+\frac{\sqrt{13}}{5}m_{2,3}+ \sum_{(i,j)\in A^*} \frac{\sqrt{i^2+j^2}}{i+j}m_{i,j}\nonumber\\[2mm]
%&=\frac{1}{\sqrt{2}}m_{2,2}+\frac{1}{\sqrt{2}}\left[5(\ell-1)+m_{2,2}-n+\sum_{(i,j) \in A^*}\left(6 \left(\frac{1}{i}+\frac{1}{j}\right)-5\right) m_{i,j}\right]\nonumber\\[2mm]
%&+ \frac{\sqrt{13}}{5}[2(n-2\ell- m_{2,2}+2)+\sum_{(i, j) \in A^*}\left(4-6 \left(\frac{1}{i}+\frac{1}{j}\right)\right) m_{i,j}]\nonumber\\[2mm]
%&+ \sum_{(i,j)\in A^*} \frac{\sqrt{i^2+j^2}}{i+j}m_{i,j}\nonumber\\[2mm]
&=\left(\sqrt{2}-\frac{2\sqrt{13}}{5}\right)m_{2,2}+ \left(\frac{2\sqrt{13}}{5}-\frac{1}{\sqrt{2}}\right)n+\left(\frac{5}{\sqrt{2}}-\frac{4\sqrt{13}}{5}\right)(\ell-1)\nonumber\\[2mm]
&+\sum_{(i,j) \in A^*}\underbrace{\left[\left(3\sqrt{2}  -\frac{6\sqrt{13}}{5}\right) \left(\frac{1}{i}+\frac{1}{j}\right) +  \frac{\sqrt{i^2+j^2}}{i+j}+\frac{4\sqrt{13}}{5}-\frac{5}{\sqrt{2}}\right]}_{h(i,j)}m_{i,j}\nonumber\\[2mm]
&\ge \left(\underbrace{\sqrt{2}-\frac{2\sqrt{13}}{5}}_{\text{negative}}\right)m_{2,2}+ \left(\frac{2\sqrt{13}}{5}-\frac{1}{\sqrt{2}}\right)n+\left(\frac{5}{\sqrt{2}}-\frac{4\sqrt{13}}{5}\right)(\ell-1),\label{Eq-main}
\end{align}
where the inequality in \eqref{Eq-main} follows from Table \ref{table-1}.

\begin{table}[h!]
\caption{Approximate values of $h(i, j)$ used in the inequality
\eqref{Eq-main}.}\label{table-1}
\renewcommand{\arraystretch}{1.5}
\centering
\begin{tabular}{|c|c|c|c|c|c|}
\hline
$(i,j)$ & $(1,3)$ & $(1,4)$ & $(2,4)$ & $(3,4)$ & $(4,4)$ \\
\hline
$h(i,j)$  & 0.0274 & 0.0685 & 0.0312 & 0.0142 & 0.0140 \\
\hline
\end{tabular}
\end{table}

Because of the definition of $G^*$, the inequality in \eqref{Eq-main} must be equality, which implies that $m_{i,j}=0$ for every $(i,j)\in A^*$. Hence, we have either $\delta(G^*)=\Delta(G^*)=3$ or $(\delta(G^*),\Delta(G^*))=(2,3)$.

If $n=2(\ell-1)$, then we can not have $(\delta(G^*),\Delta(G^*))=(2,3)$; for otherwise, we have $n_2(G^*)=n-2(\ell-1)=0$, a contradiction. Hence, if $n=2(\ell-1)$, then $G^*$ is a $3$-regular graph.

In the rest of the proof, we assume that $n>2(\ell-1)$. Then, we can not have $\delta(G^*)=\Delta(G^*)=3$; for otherwise, we have $3n=2(n+\ell-1)$ (by the Handshaking Lemma), a contradiction. Hence, $(\delta(G^*),\Delta(G^*))=(2,3)$.
Since $G^*$ is connected, it holds that $m_{2,3}>0$. Hence, from the relation $2m_{2,2}+m_{2,3}=2n_2$, we conclude that $m_{2,3}$ is a positive even integer. So, $2m_{2,2}+2\le 2m_{2,2}+m_{2,3}=2n_2=2(n-2(\ell-1))$, which yields $m_{2,2}\le n-2\ell+1$ with equality if and only if $m_{2,3}=2$. Consequently, by keeping in mind the definition of $G^*$, from \eqref{Eq-main}, we have $m_{2,2}= n-2\ell+1$, which implies $m_{2,3}=2$, and hence, $m_{3,3}= 3\ell-4$. Thus,
\begin{align*}
\mathcal{DSO}(G^*)&=\left(\sqrt{2}-\frac{2\sqrt{13}}{5}\right)(n-2\ell+1)+ \left(\frac{2\sqrt{13}}{5}-\frac{1}{\sqrt{2}}\right)n\\[2mm]
&+\left(\frac{5}{\sqrt{2}}-\frac{4\sqrt{13}}{5}\right)(\ell-1),
\end{align*}
that is,
\[
\mathcal{DSO}(G^*)=\frac{1}{\sqrt{2}}(n+\ell-3)+\frac{2\sqrt{13}}{5}.
\]

Therefore, by  \eqref{thm-1-eq-01}, the theorem follows.
\end{proof}

\section{Concluding Remarks}

One of the natural extensions of the present study is to extend
Theorem \ref{thm-1}
for all fixed-order connected graphs with a given cyclomatic number $\ell\ge3$. Particularly, if the text ``molecular'' is replaced with ``connected'' in Theorem \ref{thm-1}, then it is expected that  the resulting statement remains valid. To prove this modified statement of Theorem \ref{thm-1}, it is enough to prove the corresponding modified statement of Lemma \ref{lem-01} (which is obtained by replacing the text ``molecular'' with ``connected'' there)
because we observe that the function $h(i, j)$ used in the inequality
\eqref{Eq-main} remains positive for every $(i,j)\in A^{**}$, where  $A^{**}$ is the set obtained from $A^{*}$ by replacing the constraint $1\le i\le j\le 4$ with $1\le i\le j$.

Establishing the maximal version of Theorem \ref{thm-1} (for molecular graphs as well as for all connected graphs) is another natural open problem concerning the present study.

\acknowledgment{This study is supported via funding from Prince Sattam bin Abdulaziz University, Saudi Arabia (PSAU/2025/R/1447).}


\begin{thebibliography}{99}


\bibitem{Ali-MMN-22} A. Ali, K. C. Das, S. Akhter, On the extremal graphs for Second Zagreb Index with fixed number of vertices and cyclomatic number, {\it Miskolc Math. Notes} {\bf23} (2022) 41--50.

\bibitem{Ali-IJQC-17} A. Ali, Z. Du, On the difference between atom-bond connectivity index and Randic index of binary and chemical trees, {\it Int. J. Quantum Chem.} {\bf117} (2017) e25446.


\bibitem{Ali-AIMS-M-25} A. Ali, D. Dimitrov, T. R\'eti, A. M. Alotaibi, A. M. Alanazi, T. S. Hassan, Degree-based graphical indices of $k$-cyclic graphs, \textit{AIMS Math.} \textbf{10} (2025) 13540--13554.

\bibitem{Ali-BID-survey} A. Ali, I. Gutman, B. Furtula, Extremal results on bond incident degree indices of graphs: A survey, {\it Aequationes Math.}, submitted.

\bibitem{Ali-AIMS-25} A. Ali, I. Gutman, B. Furtula, A. M. Albalahi, A. E. Hamza, On chemical and mathematical characteristics of generalized degree-based molecular descriptors, {\it AIMS Math.} {\bf10} (2025) 6788--6804. %\url{https://doi.org/10.3934/math.2025311}

\bibitem{Ali-AIMS-2024} A. Ali, S. Sekar, S. Balachandran, S. Elumalai, A. M. Alanazi, T. S. Hassan, Y. Shang, Graphical edge-weight-function indices of trees, {\it AIMS Math.} {\bf9} (2024) 32552--32570.




\bibitem{albalahi25tricyclic} A. M. Albalahi, A. M. Alanazi, A. M. Alotaibi, A. E. Hamza, A. Ali,
        Optimizing the Euler Sombor index of (molecular) tricyclic graphs,
        {\it MATCH Commun. Math. Comput. Chem.\/} {\bf 94} (2025) 549--560.


\bibitem{Ali-IJACM-17} A. Ali, A. A. Bhatti, Z. Raza, Further inequalities between vertex-degree-based topological indices, {\it Int. J. Appl. Comput. Math.} {\bf3} (2017) 1921--1930.

%\bibitem{ali18sum} A. Ali, I. Gutman, E. Milovanovi\'{c}, I. Milovanovi\'{c}, Sum of powers of the degrees of graphs: Extremal results and bounds, {\it MATCH Commun. Math. Comput. Chem.\/} {\bf 80} (2018) 5--84.

\bibitem{Ali-CJC-16} A. Ali, Z. Raza, A. A. Bhatti, Extremal pentagonal chains with respect to bond incident degree indices, {\it Canad. J. Chem.} {\bf94} (2016) 870--876.


\bibitem{Basak-book} S. C. Basak (Ed.),
{\it Mathematical Descriptors of Molecules and Biomolecules:
Applications in Chemistry, Drug Design, Chemical Toxicology, and Computational Biology}, Springer, Cham, 2024.

\bibitem{Boesch-JGT-84} F. Boesch, R. Tindell, Circulants and their connectivities, {\it J. Graph Theory} {\bf8} (1984) 487--499.

\bibitem{Bondy-book} J. A. Bondy, U. S. R. Murty, {\it Graph Theory}, Springer, Heidelberg, 2008.

\bibitem{Bonte-25} S. Bonte, G. Devillez, V. Dusollier, A. Hertz, H. M\'elot, Extremal chemical graphs of maximum degree at most 3 for 33 degree-based topological indices, {\it arXiv}: 2501.02246v1 [math.CO], (2025).


\bibitem{borovicanin17zagreb} B. Borovi\v{c}anin, K. C. Das, B. Furtula, I. Gutman, Bounds for Zagreb indices, {\it MATCH Commun. Math. Comput. Chem.\/} {\bf 78} (2017) 17--100.


\bibitem{Chartrand-16} G. Chartrand, L. Lesniak, P. Zhang, {\it Graphs} \& {\it Digraphs}, CRC Press, 2016.



\bibitem{chen24} M. Chen, Y. Zhu, Extremal unicyclic graphs of Sombor index, {\it Appl. Math. Comput.\/} {\bf 463} (2024) 128374.

\bibitem{Cichacz-DM-16}S. Cichacz, D. Froncek, Distance magic circulant graphs, {\it Discr. Math.} {\bf339} (2016) 84--94.

\bibitem{das24open} K. C. Das, Open problems on Sombor index of unicyclic and bicyclic graphs,
        {\it Appl. Math. Comput.\/} {\bf 473} (2024) 128644.

\bibitem{KD1} K. C. Das, On the vertex degree function of graphs, {\it Comput. Appl. Math.\/} {\bf 44} (2025) 183.


\bibitem{ASO-index} K. C. Das, I. Gutman, A. Ali, Augmented Sombor Index, {\it MATCH Commun. Math. Comput. Chem.\/}, in press.

\bibitem{Desmecht-JCIM-24} D. Desmecht, V. Dubois, Correlation of the molecular cross-sectional area of organic monofunctional compounds with topological descriptors, {\it J. Chem. Inf. Model.} {\bf64} (2024) 3248--3259.

%\bibitem{Gao-MATCH-25} W. Gao, Extremal graphs with respect to vertex-degree-based topological indices for $c$-cyclic graphs, \textit{MATCH Commun. Math. Comput. Chem.} \textbf{93} (2025) 549--566.

\bibitem{gutman13degree} I. Gutman, Degree based topological indices, {\it Croat. Chem. Acta\/} {\bf 86} (2013) 351--361.

\bibitem{gutman21geo}I. Gutman, Geometric approach to degree-based topological indices: Sombor indices, {\it MATCH Commun. Math. Comput. Chem.\/} {\bf 86} (2021) 11--16.

\bibitem{gutman24eu}I. Gutman, Relating Sombor and Euler indices,
        {\it Military Tech. Courier\/} {\bf 71} (2024) 1--12.

\bibitem{gut04}I. Gutman, K. C. Das, The first Zagreb index 30 years after, {\it MATCH Commun. Math. Comput. Chem.\/} {\bf 50} (2004) 83--92.

\bibitem{gutman24elliptic} I. Gutman, B. Furtula, M. S. Oz, Geometric approach to vertex
        degree-based topological indices -- Elliptic Sombor index, theory and application,
        {\it Int. J. Quantum Chem.\/} {\bf 124} (2024) e27346.


\bibitem{Leite-24} L. S. G. Leite, S. Banerjee, Y. Wei, J. Elowitt, A. E. Clark, Modern chemical graph theory, {\it WIREs Comput. Mol. Sci.} {\bf14} (2024) e1729.

%\bibitem{Liu-MATCH-24} H. Liu, Z. Du, Y. Huang, H. Chen, S. Elumalai, Note on the minimum bond incident degree indices of $k$-cyclic graphs, \textit{MATCH Commun. Math. Comput. Chem.} \textbf{91} (2024) 255--266.

\bibitem{liu22review} H. Liu, I. Gutman, L. You, Y. Huang, Sombor index: review of extremal results and bounds, {\it J. Math. Chem.\/} {\bf 60} (2022) 771--798.


\bibitem{new11} Z. Li, L. Shi, L. Zhang, H. Ren, F. Li,
         On Sombor index of uniform hypergraphs, {\it MATCH Commun. Comput. Chem.\/}
         {\bf 94} (2025) 229--246.

\bibitem{new2} V. Maitreyi, S. Elumalai, S. Balachandran, On the
        extremal general Sombor index of trees with given pendent
        vertices, {\it MATCH Commun. Comput. Chem.\/} {\bf 92} (2024) 225--248.



\bibitem{Movahedi-26} F. Movahedi, I. Gutman, I. Red\v{z}epovi\'c, B. Furtula.
        Diminished Sombor index, {\it MATCH Commun. Comput. Chem.\/} {\bf 95}
        (2026) 141--162.

\bibitem{nithyaa24} P. Nithya, S. Elumalai, S. Balachandran, A. Ali, Z. Raza, A. A. Attiya, On unicyclic graphs with a given girth and their minimum symmetric division deg index, \textit{Discr. Math. Lett.} \textbf{13} (2024) 135--142.

\bibitem{new3} M. R. Oboudi, Sombor index of a graph and of its
        subgraphs, {\it MATCH Commun. Comput. Chem.\/} {\bf 92} (2024) 697--702.

\bibitem{Su-26} Z. Su, H. Deng, A note on the extremal graphs with respect to vertex-degree-based topological indices for $c$-cyclic graphs, \textit{MATCH Commun. Math. Comput. Chem.} \textbf{95} (2026) 163--178.


\bibitem{new7} A. P. Tache, R. M. Tache, I. Stroe, Extremal unicyclic
        graphs for the Euler Sombor index, {\it MATCH Commun. Comput. Chem.\/}
        {\bf 94} (2025) 561--578.

\bibitem{Todeschini-20-book} R. Todeschini, V. Consonni, {\it Handbook of Molecular Descriptors}, Wiley-VCH, Weinheim, 2000.

\bibitem{Trina-book} N. Trinajsti\'c, {\it Chemical Graph Theory},  CRC Press, Boca Raton, 1992.

\bibitem{Vetrik-CMB-16} T. Vetr\'ik, The metric dimension of circulant graphs, {\it Canad. Math. Bull.} {\bf60} (2017) 206--216.


\bibitem{Wagner-18} S. Wagner, H. Wang, {\it Introduction to Chemical Graph Theory},
        CRC Press, Boca Raton, 2018.


\end{thebibliography}
\end{document}